\documentclass{amsproc}

\usepackage{amsmath}
\usepackage{amsfonts}
\usepackage{amsthm}
\usepackage{graphicx}
\usepackage{eucal}
\usepackage{amscd}
\usepackage[all,2cell]{xy}
\usepackage{amssymb}
\usepackage{mathrsfs}
\usepackage{hyperref}
 \usepackage{tikz}
 \usetikzlibrary{positioning}
 \tikzset{mynode/.style={draw,circle,inner sep=1pt,outer sep=0pt}}

\numberwithin{equation}{section} 

\newtheorem{theorem}{Theorem}[section]
\newtheorem{corollary}[theorem]{Corollary}
\newtheorem{lemma}[theorem]{Lemma}
\newtheorem{proposition}[theorem]{Proposition}
\theoremstyle{definition}
\newtheorem{definition}[theorem]{Definition}
\newtheorem{Ex}[theorem]{Example}
\theoremstyle{remark}
\newtheorem{remark}[theorem]{Remark}
\newtheorem{remarks}[theorem]{Remarks}

\newtheorem{Exs}[theorem]{Examples}

\newcommand{\id}{\mbox{\rm id}}

\DeclareMathOperator{\op}{op}
\DeclareMathOperator{\Aut}{Aut}

\newcommand{\Gp}{\mathsf{Gp}}

\newcommand{\Hp}{\mathsf{Heap}}
\newcommand{\SGrp}{\mathsf{SGrp}}
\newcommand{\LQGrp}{\mathsf{LQGrp}}

\newcommand{\Set}{\mathsf{Set}}
\newcommand{\Sym}{\operatorname{Sym}}
\newcommand{\inv}{\operatorname{inv}}
\DeclareMathOperator{\E}{E}
\newcommand{\RGrp}{\mathsf{RGrp}}
\newcommand{\RHp}{\mathsf{RHp}}
\DeclareMathOperator{\bi}{+}

\newdir{ |>}{{}*!/-3.5pt/@{|}*!/-8pt/:(1,-.2)@^{>}*!/-8pt/:(1,+.2)@_{>}}

\begin{document}

\title{Right groups, left  quasigroups, and right heaps}

\author{A. Albano}
\address[Andrea Albano]{Dipartimento di Matematica  e Fisica ``Ennio De Giorgi'', Universit\`a del Salento, 73100 Lecce, Italy}
%\curraddr{Department of Mathematics and Statistics,
%Case Western Reserve University, Cleveland, Ohio 43403}
\email{andrea.albano@unisalento.it}
\thanks{A.A., M.M., and P.S.~are partially supported by the University of Salento, Department of Mathematics and Physics ``E. De Giorgi''. They are members of GNSAGA (INdAM), and members of the nonprofit association ``AGTA-Advances in Group Theory and Applications''. The first author was supported by a scholarship
financed by the Ministerial Decree no. 118/2023, based on the NRRP - funded by the
European Union - NextGenerationEU - Mission 4.}

\author{A. Facchini}
\address[Alberto Facchini]{Dipartimento di Matematica ``Tullio Levi-Civita'', Universit\`a di 
Padova, 35121 Padova, Italy}
\email{facchini@math.unipd.it}
\thanks{}

\author{M. Mazzotta}
\address[Marzia Mazzotta]{Dipartimento di Matematica  e Fisica ``Ennio De Giorgi'', Universit\`a del Salento, 73100 Lecce, Italy}
%\curraddr{Department of Mathematics and Statistics,
%Case Western Reserve University, Cleveland, Ohio 43403}
\email{marzia.mazzotta@unisalento.it}

\author{P. Stefanelli }
\address[Paola Stefanelli ]{Dipartimento di Matematica  e Fisica ``Ennio De Giorgi'', Universit\`a del Salento, 73100 Lecce, Italy}
%\curraddr{Department of Mathematics and Statistics,
%Case Western Reserve University, Cleveland, Ohio 43403}
\email{paola.stefanelli@unisalento.it}

%    General info
\date{---}


\begin{abstract}
A right group is a semigroup $(S,\cdot)$ in which, for every $a,b\in S$, there is a unique $x\in S$ such that $a\cdot x=b$. In this article, we develop the theory of heaps starting not from groups, but from right groups. We thus get a natural definition of right heap. It is even possible to develop part of the theory starting from a left quasigroup, which is the non-associative analogue of a right group. Our motivation for this study is the investigation of left non-degenerate set-theoretic solutions of the Yang--Baxter equation. Thus, we are led to an analogue of the skew left trusses introduced by T.~Brzezi\'nski.
\end{abstract}

\maketitle

\section*{Introduction} 
Over the last ten years, skew left braces have proven to be an algebraic structure extremely useful in the calculation of set-theoretic solutions of the Yang--Baxter equation \cite{GuVe17}. In our previous paper \cite{AFMS}, we observed that for the basic techniques inherent in skew left braces, for instance, in the definitions of the maps $\lambda$ and $\rho$, it is often not necessary to use the group structure, but that it is sufficient to use left subtraction and left division. That is, it is sufficient to use right groups. Furthermore, left non-degenerate solutions are related to left quasigroups, the non-associative analogues of right groups. There is thus a natural correlation between right groups, left quasigroups, left non-degenerate solutions of the Yang--Baxter equation, and the left $RG$-semibraces as defined in \cite{AFMS}. We describe this correlation in more detail, motivating our paper, in Section \ref{2} of this article. In that section, we present the basic notions and fix the notation.

One of the possible ways to present the naturalness of the skew left distributivity property $a\circ(b+c)=a\circ b-a+a\circ c$ in skew left braces is, as shown by Brzezi\'nski, that pertaining to heaps and trusses \cite{Brze}. In this way, skew left distributivity becomes left distributivity with respect to the corresponding ternary operation. The main purpose of this paper is to show that the construction of heaps and trusses is possible not only for groups but also for right groups. We thus naturally find the notion of a right heap: A {\em right heap} is a pair $(H,[-,-,-])$, where $H$ is a set and $[-,-,-]$ is a ternary operation on $H$ that is associative, {\em left Mal'tsev} (that is, $[x,x,y]=y$ for every $x,y\in H$), and {\em right weakly Mal'tsev} (i.e., $[x,y,[y,z,w]]=[x,z,w]$ for every $x,y,z,w\in H$). In this way, we extend the theory of heaps to right heaps. Indeed, in Section~\ref{4}, we show that a small portion of the theory also applies to left quasigroups, for which there is a corresponding natural notion of a left quasiheap. In short, even for right groups and left quasigroups, it is possible to adopt the viewpoint of Viktor V.~Wagner \cite{HL}.

A heap can be thought of as a group with the identity element ``forgotten''.  In the case of right groups, we don't have an identity but several left identities, and in right heaps we forget all left identities, finding an extremely uniform structure. Particularly interesting is what the product decompositions of a right group as a product of a group and a right zero semigroup become in the case of right heaps.

\smallskip

In this paper, when we say ``algebra'' or ``variety'' we mean in the sense of Universal Algebra.

\section {Basic notions}\label{2}
\subsection{Left quasigroups} For every magma~$(M,\,\cdot)$, that is, for every set $M$ with a binary operation $\cdot$, 
it is possible to define a map~$\ell\colon  M\to M^M,\ m \mapsto \ell_m$,  where 
$\ell_m(n)=m\cdot n$ for all $n\in M$. This map  $\ell\colon  M\to M^M$ is a magma morphism if and only if $(M,\,\cdot)$ is a semigroup. The magma $(M,\,\cdot)$ is  {\em left cancellative} if all the mappings $\ell_m\colon M\to M$ ($m\in M$) are injective; \emph{right simple} if all the mappings $\ell_m\colon M\to M$ are surjective;
a {\em left quasigroup} if all the mappings $\ell_m\colon M\to M$ are bijective. If $(M,\,\cdot)$ is a left quasigroup, all the mappings $\ell_m\colon M\to M$ have inverse mappings  $(\ell_m)^{-1}\colon M\to M$, hence it is possible to define another binary operation $\backslash $ on the set $M$ setting $m\backslash n := (\ell_m)^{-1}(n)$ for all $m,n\in M$. Since $\ell_m$ and $(\ell_m)^{-1}$ are mutually inverse mappings $M\to M$, it follows that $m\cdot (m\backslash n) = n = m\backslash (m\cdot n)$ for all $m,n\in M$. Hence, the variety of left quasigroups can be defined as the variety of all algebras $(Q,\,\cdot,\,\backslash)$, where $Q$ is a set and $\cdot,\backslash$ are two binary operations on $Q$ satisfying the identities $x\cdot (x\backslash y)=y$ and $x\backslash (x\cdot y)=y$. In the setting of left quasigroups, we will say that the binary operation $\backslash$ is the {\em right inverse} of the binary operation $\cdot$. The right inverse of multiplication  $\cdot$ is usually called {\em right division}. In additive notation, the right inverse of addition $+$ is usually called {\em right subtraction}. 
Our first lemma follows immediately from the notions above.
\begin{lemma}
The following conditions are equivalent for a magma $(M,\,\cdot)$:
\begin{itemize}
    \item[{\rm(a)}] $M$ is a left quasigroup.
    \item[{\rm(b)}] 
    $M$ is right simple 
and left cancellative.
\item[{\rm(c)}] For every $a,b\in M$, there exists a unique element $x\in M$ such that $a \cdot x = b$.\end{itemize}\end{lemma}

\begin{remark}\label{1} If, in the two identities that define the variety of left quasigroups, we exchange the two operations $\cdot$ and $\backslash$, the two identities swap with each other. 
As a consequence, if
$(Q,\cdot)$ is a left quasigroup, then $(Q,\backslash)$ is also a left
quasigroup. \end{remark}
For every left quasigroup $(M,\,\cdot)$, the map $\ell$ can be seen as a map $\ell\colon  M\to \Sym_M$, the symmetric group on the set $M$. Clearly, an element $m\in M$ is a left identity if and only if $\ell_m=\id_M$, the identity mapping $\id_M\colon M\to M$. Also, every left identity of $M$ is obviously an idempotent element.

It is now clear that all left quasigroups can be constructed as follows. 
Fix a set $X$ and any map $\varphi\colon X\to\Sym_X$. Define a multiplication $\cdot$ on $X$ setting $x\cdot y:=\varphi(x)(y)$ for all $x,y\in X$. Then $(X,\cdot)$ is a left quasigroup, and all left quasigroups can be constructed in this way. The right inverse operation $\backslash$ is given by $x\backslash y:=(\inv\circ \varphi)(x)(y)$, where $\inv\colon\Sym_X\to\Sym_X$ is the group anti-automorphism~$\sigma\mapsto\sigma^{-1}$.

Given a left quasigroup $(M,\,\cdot)$, let $f\colon M\to M$ be the mapping defined by $f(x)=x\backslash x$ for every $x\in M$, and let $E$ be the image of $f$. Let $\pi_E$ denote the corestriction of $f$ to $E$, that is, the mapping $\pi_E\colon M\to E$ defined by $\pi_E(x)=x\backslash x$ for every $x\in M$. Then $\{\,\pi_E^{-1}(e)\mid e\in E\,\}$ is a partition of $M$, where, for each $e\in E$, 
$$
\pi_E^{-1}(e)=\{\,x\in M\mid \pi_E(x)=e\,\}=\{\,x\in M\mid x\cdot e=x,\}.
$$ 
Moreover, for every idempotent element $e\in M$, $e\cdot e=e$ implies that $e\in \pi_E^{-1}(e)$, that $e$ acts as a right identity on $\pi_E^{-1}(e)$, and that $e\in E$.

Thus, for a left quasigroup $(M,\,\cdot)$, we have three naturally defined subsets of $M$: the set $E_1$ of all left identities of $(M,\,\cdot)$, the set $E_2$ of all idempotents of $(M,\,\cdot)$, and the set $E$ defined in the previous paragraph, and we have seen that $E_1\subseteq E_2\subseteq E$. Let us give an example that shows that both these inclusions can be proper.

\begin{Ex}\label{www} Let $X=\{1,2,3\}$ be a set, 
$S_3$ be the symmetric group on three objects, $\tau\in S_3$ be the transposition $(1\ 2)$, and $\varphi\colon X\to S_3$ be the constant mapping equal to~$\tau$. Then the corresponding left quasigroup $X$ has no left identity, $3$ is the unique idempotent element, and $E_3=X$. The partition in the sets $\pi_E^{-1}(e)$ is the trivial partition $\{\{1\},\{2\},\{3\}\}$.
\end{Ex}

Similarly for {\em right quasigroups}, the magmas~$(M,\,\cdot)$ for which the maps 
$r_m(n)=n\cdot m$ are all biective. If $(Q, \cdot)$ is both a left quasigroup and a right quasigroup, then $(Q, \cdot)$ is called a {\em quasigroup}.

\subsection{Right groups} {\em Right groups} are the non-empty left quasigroups $(Q,\cdot)$ with $\cdot$ associative, that is, the non-empty left quasigroups that are semigroups. 

A {\em right zero semigroup}  \cite[p.~4]{ClPr61} is a semigroup $(Q,\cdot)$ in which
$a\cdot b=b$ for all $a,b\in~Q$. Since the operation $\cdot$ 
coincides with the second canonical projection~$\pi_2\colon S\times S\to S$, we will denote the operation itself by $\pi_2$, writing $a\,\pi_2\, b=b$ for all $a,b\in S$. 
Hence, right zero semigroups are those of the form $(S,\pi_2)$ for some set $S$. In a magma $M$ all elements are left identities if and only if $M$ is a right zero semigroup, if and only if the map $\ell\colon M\to M^M$ is the constant mapping equal to $\id_M$.

\begin{lemma}\label{1.2} {\rm \cite[Lemma~1.26]{ClPr61}}  Every idempotent of a right simple semigroup $S$ is a left identity for $S$.
\end{lemma}

%\begin{proof} If $e,x\in S$ and $e^2=e$, then $x=ey$ for some $y\in S$, so $ex=e(ey)0ey=x$.
%\end{proof}

From Lemma \ref{1.2} we see that $E_1=E_2=E$ for a right group $S$. Thus Lemma~\ref{1.2} holds for right groups, but not for left quasigroups, as the left quasigroup in Example~\ref{www} shows.

\begin{theorem}\label{1.1} 
{\rm \cite[Section~1.11, Theorem~1.27]{ClPr61}} The following assertions on a semigroup $S\ne\emptyset$ are equivalent:
\begin{itemize}
\item[{\rm(a)}] $S$ is a right group.

\item[{\rm(b)}] $S$ is right simple and contains an idempotent.

\item[{\rm(c)}] $S$ is isomorphic to the external direct product of a group $G$ and a non-empty right zero semigroup $E$.

\item[{\rm(d)}] There exists an element $e\in S$ such that: \\{\rm (1)}  $e$ is a left identity  for $S$; and\\
{\rm(2)} every element of $S$ has a right inverse with respect to $e$.

\item[{\rm(e)}] {\rm (1)} $S$ has a left identity; and \\
{\rm (2)} for every left identity $e$ of $S$ and every element $a\in S$, $a$ has a right inverse with respect to $e$.

\item[{\rm(f)}]  $S$ is the disjoint union of a non-empty family of groups, and the set of the identities of these groups is a right zero subsemigroup of $S$.
\end{itemize}
\end{theorem}

%Similarly to Lemma \ref{1.2}, every idempotent of a left cancellative semigroup $S$ is a left identity for $S$, because if $e,x\in S$ and $e^2=e$, then $e(ex)=ex$ implies $ex=x$ for $S$ left cancellative. But $S$ left cancellative and $S$ contains an idempotent does not imply ``$S$ is a right group'', as the example of the additive monoid $\N_0$ shows. 
%
%\smallskip

Notice that the symmetry between the two operations $\cdot$ and $\backslash$ described in Remark~\ref{1} does not hold for associativity. Therefore if $(S,\cdot)$ is a right group, then $(S,\backslash)$ is a left
quasigroup, but not a right group in general, because associativity is missing. For instance, if $(G,\cdot)$ is a group, the operation $\backslash$ is the operation $a\backslash b=a^{-1}b$, hence $(G,\backslash)$ is a left quasigroup,
but not a right group in general. More precisely we have:

\begin{lemma} The following conditions are equivalent for a left quasigroup $(S,\cdot,\backslash)$:
\begin{itemize}
\item[\rm (1)] $S$ is a right group.

\item[\rm (2)] The mapping $\ell\colon S\to\Sym_S$, $\ell_x=x\cdot-$, is a semigroup homomorphism of the semigroup $(S,\cdot)$ into the group $\Sym_S$.

    \item[\rm (3)] The mapping $\ell'\colon S\to\Sym_S$, $\ell'_x=x\backslash-$, is a semigroup antihomomorphism of the semigroup $(S,\cdot)$ into the group $\Sym_S$.

\item[\rm (4)] $x\backslash(y\backslash z)=(y\cdot x)\backslash z$ for every $x,y,z\in S$.

\item[\rm (5)] The magma $(S,\backslash)$ is isomorphic to the direct product $G\times E$, where $G$ is a group with the operation $x\backslash y=x^{-1}y$, and $E$ is a right zero semigroup.
\end{itemize}
\end{lemma}

\subsection{Set-theoretic solutions of the Yang--Baxter equation}\label{page5}
Let $X$ be a set and consider a mapping
$r\colon X\times X\to X\times X$. The pair $(X, r)$ is a \emph{set-theoretic solution of the Yang--Baxter equation}, briefly a \emph{solution}, if 
\begin{align}\tag{YBE}\label{YBE}
\left(r\times\id_X\right)
\left(\id_X\times \, r\right)
\left(r\times\id_X\right)
= 
\left(\id_X\times \, r\right)
\left(r\times\id_X\right)
\left(\id_X\times \, r\right).
\end{align}
We will write $r\left(x, y\right) = \left(\lambda_{x}\left(y\right), \rho_{y}\left(x\right)\right)$, where, for all $x,y\in X$, $\lambda_{x}$ and $\rho_{y}$ are maps $X\to X$. The pair
$(X, r)$ is a solution if and only if, for all $x,y,z \in X$,
    \begin{align}
     &\label{first} \lambda_x\lambda_y(z)=\lambda_{\lambda_x\left(y\right)}\lambda_{\rho_y\left(x\right)}\left(z\right)\tag{Y1}\\
    &  \label{second}\lambda_{\rho_{\lambda_y\left(z\right)}\left(x\right)}\rho_z\left(y\right)=\rho_{\lambda_{\rho_y\left(x\right)}\left(z\right)}\lambda_x\left(y\right)\tag{Y2}\\
      &\label{third}\rho_z\rho_y(x)=\rho_{\rho_z\left(y\right)}\rho_{\lambda_y\left(z\right)}\left(x\right).\tag{Y3}
  \end{align} 
A solution is \textit{left non-degenerate} if $\lambda_x \in \Sym_X$ for all $x \in X$; 
\textit{right non-degenerate} if $\rho_x \in \Sym_X$ for all $x \in X$; and \emph{non-degenerate} if it is both left and right non-degenerate. 

This can be also stated in the language of Universal Algebra. A 
solution 
is an algebra $(X, \cdot,\diamond)$ where $\cdot$ and $\diamond$ are binary operations satisfying the three identities
\begin{align}
     &\label{first'}
   x\cdot (y \cdot z)=(x\cdot y)\cdot((x\diamond y)\cdot z)\tag{Y'1}\\
    &  \label{second'} (x\cdot y)\diamond ((x\diamond y)\cdot z)=(x\diamond (y\cdot z))\cdot(y\diamond z)  \tag{Y'2}\\
      &\label{third'} (z\diamond y)\diamond x=(z\diamond(y\cdot x))\diamond(y\diamond x)\tag{Y'3}
  \end{align} 
for all $x, y, z \in X$. 
This is clearly related to biracks, see \cite{JePiZa19}, \cite{St06}, in the general setting as in \cite{St25}. 
Hence,  solutions 
form a variety. In this language, it is obvious that a solution $(X, \cdot,\diamond)$ 
is left non-degenerate if and only if $(X, \cdot)$ is a left quasigroup, and is right non-degenerate if and only if $(X, \diamond)$ is a right quasigroup. Moreover, if $(X, \cdot,\diamond)$ is a solution and either $(X, \cdot)$ or $(X, \diamond)$ are quasigroups, then $\cdot$ and $\diamond$
uniquely determine each other.

\smallskip

Special solutions 
can be studied and characterized in different ways, e.g., making use of particular algebraic structures. For instance, this occurs for the solutions~$(X,r)$ for which both mappings $\lambda\colon X\to X^X$ and $\rho\colon X\to X^X$ are constant mappings, that is, there exist $\sigma,\tau\colon X\to X$ such that $\lambda_x(y)=\sigma(y)$ and $\rho_y(x)=\tau(x)$ for all $x,y\in X$. Cf.~\emph{Lyubashenko's solutions} or \emph{permutation solutions} \cite[Example 2]{ESS99}.

\smallskip

Recall that two binary operations $\cdot,\diamond$ on a set $X$ satisfy the {\em interchange property} if 
$$
(x\cdot y)\diamond(w\cdot z)=(x\diamond w)\cdot(y\diamond z)\quad \mbox{\rm for every }x,y,w,z\in X.
$$ 
See \cite{Australian}.

\begin{proposition}\label{equiv-const-sol}
Let $X$ be a set and $\sigma,\tau\colon X\to X$ be two mappings. The following conditions are equivalent:
\begin{itemize}
    \item[\rm (a)] There exists a 
solution $(X,r)$ 
for which $\lambda_x(y)=\sigma(y)$ and $\rho_y(x)=\tau(x)$, for all $x,y\in X$. 

\item[\rm (b)] $\left(X,(\sigma\pi_2)\times(\tau\pi_1)\right)$ is a %set-theoretic solution of the Yang-Baxter equation.
solution.

\item[\rm (c)] $\sigma\tau = \tau\sigma$.

\item[\rm (d)] The binary operations $\sigma\pi_2$ and $\tau\pi_1$ satisfy the interchange property.
\end{itemize}
\end{proposition}

\begin{proof} (a)${}\Leftrightarrow{}$(b) is trivial, because $\lambda_x=\sigma$ and $\rho_y=\tau$ if and only if $r\colon X\times X\to X\times X$ is equal to the product mapping $(\sigma\pi_2)\times(\tau\pi_1)$.

(a)${}\Leftrightarrow{}$(c) Replacing in (a) $\lambda_x(y)$ with $x\cdot y$ and $\rho_y(x)$ with $x\diamond y$, we have that (a) is equivalent to %(Y'1), (Y'2), (Y'3). 
\eqref{first'}, \eqref{second'}, \eqref{third'}. Now $x\cdot y=\sigma(y)$ and $x\diamond y=\tau(x)$, so that 
\eqref{first'} and \eqref{third'} always hold. Moreover 
\eqref{second'} can be rewritten as $\tau(\sigma(y))=\sigma(\tau(y)))$ for all $y\in X$, i.e., $\sigma$ and $\tau$ commute.
Therefore,  (a) is equivalent to (c).

(c)${}\Leftrightarrow{}$(d) The interchange property holds for $\sigma\pi_2$ and $\tau\pi_1$ if and only if $\sigma$ and $\tau$ commute.
\end{proof}

\begin{Ex}{\rm 
    Some further examples of %set-theoretic 
    solutions 
    %of the Yang-Baxter equation 
    can be constructed very easily. For instance, fix a set $X$ and a bijection $f\colon X\to X$. Then the mapping $r\colon X\times X\to X\times X$, defined by
    $r(a,b)=(f(b),b)$, for every $a,b\in X$, is a left non-degenerate 
    idempotent solution, namely, $r^2=r$. 
    
    To prove it, check that the three axioms %(Y'1), (Y'2) and (Y'3) 
    \eqref{first'}, \eqref{second'} and \eqref{third'} are satisfied for $a\cdot b=f(b)$ and $a\diamond b=b$, for all $a,b\in X$.}
\end{Ex}

\subsection{Left \texorpdfstring{$RG$-}{-}semibraces}\label{1.4}
  Let us briefly recall the main results obtained in \cite{AFMS}.  A {\em left $RG$-semibrace} is a triplet $(B,+,\circ)$, where $B$ is a set and $+,\circ$ are binary operations on $B$, such that both $(B,+)$ and $(B,\circ)$ are right groups 
and the identity
\begin{equation*} 
a\circ (b+ c) = a\circ b+(-a+ a\circ c)\label{bjip}
\end{equation*} 
holds for all $a,b,c\in B$, where $-a+ a\circ c$ is the unique element $x$ of $B$ such that $a+x=a\circ c$. That is, $-a+ a\circ c$ is  $ a\backslash (a \circ c)$, where $\backslash$ is the right inverse of the operation $+$ of the right group $(B,+)$. 

For a left $RG$-semibrace $(B,+,\circ)$, define a mapping $\lambda\colon B\to\Aut_\SGrp(B,+)$ setting $\lambda_a(b)=a\backslash (a \circ b)$, for all $a,b\in B$. Here $\SGrp$ is the category of semigroups. This mapping $\lambda$ turns out to be a semigroup homomorphism of the right group $(B,\circ)$ into the group $\Aut_\SGrp(B,+)$.

In a left $RG$-semibrace $(B,+,\circ)$, the set $F$ of all idempotents of the right group $(B,\circ)$ is always contained in the set $E$ of all idempotents of the right group $(B,+)$. There is a canonical projection $\pi_E\colon B\to E$ defined by $\pi_E(a)=a\backslash a$, for all $a\in B$. The projection $\pi_E$ is a semigroup homomorphism of the right group $(B,+)$ onto its subsemigroup $E$.

If we denote by $\backslash_\circ$ the right inverse of the operation $\circ$, we get a canonical projection $\pi_F\colon B\to F$ defined by $\pi_F(a)=a\backslash_\circ a$, for all $a\in B$, and $\pi_F$ turns out to be a semigroup homomorphism of the right group $(B,\circ)$ onto its subsemigroup~$F$. Moreover, in this notation, define 
 $\rho_b(a)=(\lambda_a(b))\backslash_\circ(a \circ b)$, for every $a,b\in B$. Then:

 \begin{theorem}\label{t}{\rm \cite[Theorem 3.5]{AFMS}}
    Let $B$ be a left $RG$-semibrace. The map $r_B: B \times B \to B \times B$, defined by
    \begin{align*}
        r_B(a,b)=\left(\lambda_a(b), \rho_b(a)\right)
    \end{align*}
    for all $a, b \in B$, is a left non-degenerate solution. 
    %of the Yang-Baxter equation. 
    \end{theorem}

    For any binary operation $\circ$ on a set $X$, it is possible to define its opposite operation $\circ^{\op}$ by setting $a\circ^{\op}b:=b\circ a$, for all $a,b\in X$.
Of course, if $(X, \cdot,\diamond)$ is an algebra in which $\cdot$ and $\diamond$ are binary operations that satisfy the three identities %(Y'1), (Y'2) and (Y'3), 
\eqref{first'}, \eqref{second'} and \eqref{third'}, then $(X,\diamond^{\op}, \cdot^{\op})$ also 
satisfies %(Y'1), (Y'2) and (Y'3)
\eqref{first'}, \eqref{second'} and \eqref{third'}. 
There is an involutory category automorphism of the category of %set-theoretic 
solutions that sends each solution $(X, \cdot,\diamond)$ to the solution $(X,\diamond^{\op}, \cdot^{\op})$ and is the identity on solution morphisms.  Notice that axioms (%Y'1) and (Y'3) 
\eqref{first'} and \eqref{third'} correspond in this involutory category automorphism. Also, notice that axiom 
\eqref{first'} is exactly the left associativity of gyrogroups (left gyroassociative property), similarly \eqref{third'} is a right gyroassociative property, and \eqref{second'} is a sort of ``gyrointerexchange property''.

\begin{corollary}
    If $B$ is a left $RG$-semibrace, then the map $r'_B: B \times B \to B \times B$, defined by
    \begin{align*}
        r'_B(a,b)=\left(\rho_a(b), \lambda_b(a)\right)
    \end{align*}
    for all $a, b \in B$, is a right non-degenerate solution. %of the Yang-Baxter equation. 
    \end{corollary}

Clearly, there is also an involutory category automorphism of the category of bijective set-theoretic solutions that sends each solution $(X, r)$ to the solution $(X,r^{-1})$.

\smallskip

Notice the relation between the concepts of right groups, left quasi-groups, left non-degenerate solutions, and right inverse operations. If $(B,+,\circ)$ is a left $RG$-semibrace and the operation $\cdot$ is defined by $a\cdot b=\lambda_a(b)$, for every $a,b\in B$, then it follows from \cite[Proposition 3.5]{AFMS} that $(B,\cdot)$ is a left quasigroup.
%To see it, fix $a$ and $b$ in $B$. Then $x$ is an element of $B$ such that $a\cdot x=b$ if and only if $-g_a+a\circ x=b$, if and only if $a\circ x=g_a+b$. Such an $x$ in $B$ exists and is unique because $(B,\circ)$ is a right group.
The definition of $\rho$ is
$\rho_a(b)=(a\cdot b)\backslash_\circ (a\circ b)$ in the right group $(B,\circ)$. 
Conversely, we have the following result.

\begin{proposition}\label{questa} Let $(B,+,\circ,\cdot,\diamond)$ be an algebra (in the sense of Universal Algebra) of type $(2,2,2,2)$ {\rm(}that is, the four operations $+,\circ,\cdot$ and $\diamond$ are all binary{\rm)}, satisfying the following conditions {\rm(}identities{\rm)}:
\begin{itemize}
\item[{\rm(a)}] $x+(x\cdot y)=x\circ y$ for all $x,y \in B$;

\item[{\rm(b)}] the three identities 
%{\rm (Y'1)}, {\rm  (Y'2)} and {\rm (Y'3)} 
\eqref{first'}, \eqref{second'} and \eqref{third'} hold;

\item[{\rm (c)}] $(x\cdot y)\circ(x\diamond y)=x\circ y$ for all $x,y\in B$.
\end{itemize}
\noindent If $(B,+)$ is a left quasigroup and the solution $r(x,y)=(x\cdot y,\,x\diamond y)$ is  left non-degenerate, then $(B,\cdot)$ and $(B,\circ)$ are also  left quasigroups, and the operations $\cdot$ and $\diamond$ are completely determined by the operations $+$ and $\circ$.
\end{proposition}
\begin{proof} The fact that the solution $r$ is left non-degenerate is equivalent to the fact that $(B,\cdot)$  is a left quasigroup.

The operation $\cdot$ is completely determined by the operations $+$ and $\circ$ by the identity in (a) and the fact that $(B,+)$ is a left quasigroup.

Let us prove that $(B,\circ)$ is a left quasigroup. Let $a,b$ be two elements of $B$. We want to count the elements $x\in B$ such that $a\circ x=b$. If $x$ is such that $a\circ x=b$, then $a+(a\cdot x)=b$. Since $(B,+)$ is a left quasigroup, there is only one element $c\in B$ such that $a+c=b$.  Since $(B,\cdot)$  is a left quasigroup, there is a unique element $x\in B$ such that $a\cdot x=c$. This proves that there exists a unique element $x\in B$ such that $a\circ x=b$: if $c\in B$ is the unique element such that $a+c=b$, then $x\in B$ is the unique element such that $a\cdot x=c$.  
Therefore, $(B,\circ)$ is a left quasigroup. 

Now, from (c) and what we have proved in the previous two paragraphs, it follows that
the operation $\diamond$ is completely determined by the operations $+$ and~$\circ$. 
\end{proof}

\begin{remark}
In the notations of Proposition~\ref{questa}, if $(B,\cdot)$ is a quasigroup, there is a unique binary operation $\diamond$ on $B$ for which the map $r\colon B\times B\to B\times B$ defined by $r(x,y)=(x\cdot y, \, x\diamond y)$ is a solution, by %(Y'1). 
\eqref{first'}. 
This solution $r$ is necessarily left non-degenerate. Dually, if $(B,\cdot)$ is a quasigroup, there is a unique binary operation $\cdot$ on $B$ for which $r\colon B\times B\to B\times B$ defined by $r(x,y)=(x\cdot y, \, x\diamond y)$ is a solution, by %(Y3). 
\eqref{third'}. This solution $r$ is necessarily right non-degenerate. 
\end{remark}

Given a semigroup $(S,\cdot)$, we will denote by $\E(S,\cdot)$ the set of all idempotents of~$S$. 
Notice that 
the direct-product representation of a right group $S$ as the direct product of a group $G$ and a non-empty right zero semigroup $E=\E(S,\cdot)$ in 
Theorem~\ref{1.1}(c)
is rather particular because 
$G$ and $E$ are both isomorphic to homomorphic images and subsemigroups of a right group $S$. Namely, %if $\sigma_e$ is the idempotent endomorphism of $S$ defined by setting $\sigma_e(s) = s \cdot e$, for all $s \in S$ and every idempotent $e \in S$, then the image of $\sigma_e$ is the subgroup $S \cdot e$ of $S$ that is isomorphic to $G$.
for every idempotent $e$ of $S$, there is an idempotent endomorphism $\sigma_e$ of the semigroup $S$  defined by $\sigma_e(s)=s\cdot e$, for every $s\in S$, whose image is the subgroup $S\cdot e$ of $S$, with $S\cdot e\cong G$. 
In the same way, if $\ell_s:S\to S$ is the map defined by $\ell_s(a) = s\cdot a$, for all $a,s\in S$, by Remark~\ref{1.1}(b), there is an idempotent endomorphism $\tau$ of the semigroup $S$ defined by $\tau(s)=s\backslash s$, for every $s\in S$, and the image of $\tau$ is the subsemigroup $E$ of $S$.  Images of idempotent semigroup endomorphisms of a semigroup $S$ are at the same time subsemigroups of $S$ and homomorphic images of $S$. 
Also, if $\Delta\colon S\to S\times S$ and $\mu\colon S\times S\to S$ are the mappings defined by $\Delta(s) = (s,s)$ and $\mu(s,s')=s\cdot s'$, for all $s,s'\in S$, respectively, then 
$
\mu(\sigma_e\times\tau)\Delta=\id_S,
$
for every $e\in E$. Moreover, $\sigma_e\tau=\tau\sigma_e=c_e$, where $c_e\colon S\to S$ is the mapping constantly equal to $e$. In categorical language, $S$ is the product of $G$ and $E$ in the category of semigroups, and $(S,\cdot, e)$ is the coproduct of $(G\cdot e,\cdot, e)$ and $(E,\pi_2,e)$ in the category of pointed right groups (see Definition~\ref{deff}). %Cf.~\cite{FacFin}. 
Thus $(S,\cdot, e)$ is the biproduct of $(G\cdot e,\cdot, e)$ and $(E,\pi_2,e)$ in the category of pointed right groups. Note that, in general,  $\mu(\sigma\times\tau_e)\Delta$ is different from $\id_S$.

\smallskip

Some care is necessary in dealing with congruences and homomorphisms of left quasigroups and right groups. For a right group $(S,\,\cdot)$, it is possible to factor out $S$ modulo any equivalence relation $\sim$ on $S$ compatible with the operation $\cdot$ getting a semigroup $S/{\sim}$. If $(S,\,\cdot)$ is a right group and $\sim$ is an equivalence relation on $S$ compatible with both the operation $\cdot$ and its right inverse operation $\backslash$, then $S/{\sim}$ is a right group.

As far as morphisms are concerned, the category $\RGrp$ of right groups is a full subcategory of the category $\SGrp$ of semigroups, that is, every semigroup morphism between right groups also respects the operation $\backslash$, i.e., for all right groups $S,S'$ and all semigroup morphism $f\colon S\to S'$, one necessarily has that $f(s_1\backslash s_2)=f(s_1)\backslash f(s_2)$, for every $s_1,s_2\in S$.
In the category $\LQGrp$ of left quasigroups the morphisms are the mappings compatible with both the operations~$\cdot$ and~$\backslash$. By Remark~\ref{1}, there is an involutory category isomorphism $\LQGrp\to \LQGrp$ that sends each left quasigroup $(Q, \cdot,\backslash)$ to the left quasigroup $(Q, \backslash, \cdot)$ and is the identity on left quasigroup morphisms.

For any right group $(S,\,\cdot,\,\backslash)$, let $E$ denote the set of all idempotents of $S$, and
 fix an element $e_0\in E$. Define an equivalence relation $\sim $ on $S$ setting $a\sim b$ if and only if $a\cdot e_0=b\cdot e_0$, for every $a,b\in S$. The equivalence $\sim$ does not depend on the choice of the idempotent $e_0$, since, for every $a,b\in S$ and every $e\in E$, one has that $a\cdot e=b\cdot e$ if and only if $a\cdot e_0=b\cdot e_0$. 
The equivalence $\sim$ is compatible with both the operations $\cdot$ and $\backslash$ on $S$. 
The equivalence class of any $a\in S$ modulo $\sim$ is $a\cdot E:=\{\,a\cdot e\mid e\in E\,\}$, so that there is a partition $\{\,a\cdot E\mid a\in S\,\}$ of $S$. A complete irredundant set of representatives of the congruence classes of $S$ modulo $\sim$ is the set $S\cdot e_0:=\{\,a\cdot e_0\mid a\in S\,\}$, which is a subgroup of $(S,\,\cdot)$ with identity $e_0$. Moreover, $S\cdot e_0$ and $S/{\sim}$ are canonically isomorphic groups.

\smallskip

For our computations in the next two sections, given a right group $(S,+)$, it is convenient to fix an idempotent element in $S$, and denote it by $0$. Thus, as we have already said, in the category of pointed semigroups, $(S,+,0)$ is the coproduct of its pointed subsemigroups $(S+0,\cdot,0)$, which is a group, and $(E,\pi_2,0)$, the set of all idempotent elements in $(S,+)$, which is a right zero semigroup. Moreover, we can write any element $b\in S$ in a unique way as $$b = g_b + e_b,$$
where $g_b=b+0\in S+0$ and $e_b=b\backslash b\in E$. Here, $\backslash$ is the right inverse operation of the operation $+$ of the right group $(B,+)$. 

\section{Right heaps}

For a left $RG$-semibrace $(B,+,\circ)$, we have seen that  skew left distributivity\linebreak $a\circ (b+c)=a\circ b+(a\backslash (a\circ c))$, the definition $\lambda_a(b)=a\backslash (a\circ b)$, and the definition $\rho_b(a)=(\lambda_a(b))\backslash_\circ (a\circ b)$ are expressed making use of the operations $+$ and $\circ$ of the left $RG$-semibrace and their right inverse operations $\backslash$ and $\backslash_\circ$. 

Recall that a {\em skew left truss} \cite[Definition 2.1]{Brze} is a quadruple $(A,+,\circ,\sigma)$ where $(A,+)$ is a group, $(A,\circ)$ is a semigroup, $\sigma\colon A\to A$ is a unary operation, and the {\em truss distributive law} is satisfied, namely,
$$
a\circ(b+c)=a\circ b-\sigma(a)+a\circ c
$$ 
for all $a,b,c\in A$. Now (see \cite[p.~4151 and Remark~2.6]{Brze}), 
the truss distributive law is equivalent to the distributive law
$$
a \circ [b, c, d] = [a\circ b,\, a \circ c,\, a \circ d]
$$ 
for all $a, b, c, d\in A$.
Moreover (see \cite[Remark~1.5]{AFMS}), if $B$ is a left $RG$-semibrace and $f$ is any element of $B$ with $f\circ f=f$, then $f+f=f$, so that the mapping $\sigma_f\colon B\to B$, given by $\sigma_f(a)
= a + f$ for all $a\in B$, is an idempotent endomorphism of the additive right group $(B,+)$. In this notation, skew left distributivity can be rewritten as 
$$
a\circ (b+ c) = a\circ b -\sigma_f(a) + a\circ c,
$$ 
for all $a,b,c\in B$, where $y:= -\sigma_f(a) + a\circ c$ is the unique element in $B$ such that $\sigma_f(a) + y = a\circ c$. Thus, the definition of left $RG$-semibrace is in relation to that of skew left trusses, hence with heaps \cite[Section~2]{Brze}. This leads us to the study of {\em right} heaps, the analogues for right groups of what heaps are for groups. It is then natural to ask whether there is a relation between left $RG$-semibraces and {\em left $RG$-semitrusses}, as in the case of skew left trusses \cite[Theorem~2.5]{Brze}. We will see in Theorem~\ref{2.14} that this is the case. 

\smallskip

Let us briefly recall what the situation is for groups and heaps. 
\begin{definition}
    An algebra $(H, [-,-,-])$, where  $[-,-,-]\colon H\times H\times H\to H$  is a  ternary operation on $H$, is said to be a {\em heap} if the following identities are satisfied:
    \begin{itemize}
        \item[(1)] $[x,x,y] = y$ and $[x,y,y] = x$ for every $x,y\in H$, i.e., $[-,-,-]$ is a {\em Mal'tsev operation},
        
        \item[(2)] $[[x,y,z],w,u]=[x,y,[z,w,u]]$ for every $x,y,z,w,u\in H$, i.e., $[-,-,-]$ is an {\em associative} ternary operation.
    \end{itemize}
\end{definition}

One of the best examples of a heap is, for any two algebras $A$ and $B$ in a variety, the set of all isomorphisms from $A$ to $B$ under the ternary operation $[f,g,h]=fg^{-1}h$, e.g., the set of all bijections of a set $X$ onto a set $Y$. 
This is really one of the best examples, first of all because the set of all isomorphisms from $A$ to $B$ can be empty; moreover, because there is no canonical isomorphism from $A$ to $B$, there is no identity isomorphism from $A$ to $B$.

 %\smallskip
 
 Heaps form a variety in the sense of Universal Algebra, hence a category $\Hp$ in which morphisms $(H,[-,-,-])\to(H',[-,-,-])$ are the mappings $f\colon H\to H'$ for which $f([x,y,z)=[f(x), f(y), f(z)]$ for all $x,y,z\in H$. 

It is often convenient to replace a ternary operation $[-,-,-]$ on a set $X$ with an indexed family $\{\,{\bi}_y\mid y\in X\,\}$ of binary operations ${\bi}_y\colon X \times X\to X$, defined by $x{\bi}_yz=[x,y,z]$ for every $x,y,z\in X$. Correspondingly, we get a family of magmas $(X,{\bi}_y)$, which is indexed in $X$ itself. 
It is easily seen that a ternary operation $[-,-,-]$ on a set $X$ is a Mal'tsev operation if and only if, for the corresponding indexed family $\{\,{\bi}_y\mid y\in X\,\}$ of binary operations, the element $y$ is a two-sided identity of the magma $(X,{\bi}_y)$, for every $y\in X$. Notice that, in a magma, a two-sided identity, when it exists, is unique. Also, if  a ternary operation $[-,-,-]$ on a set $X$ is associative, then all the binary operations ${\bi}_y$ are associative, that is, all the magmas $(X,{\bi}_y)$ are semigroups. Finally, if $(X,[-,-,-])$ is a heap, then all the monoids $(X,{\bi}_y)$ are pair-wise isomorphic groups. The group isomorphisms $(X,{\bi}_x)\to (X,{\bi}_y)$ are the mappings $\tau_x^y\colon (X,{\bi}_x)\to (X,{\bi}_y)$ defined by $\tau_x^y(z)=[z,x,y]$, for every $x,y,z\in X$.

There is a canonical functor $U\colon \Gp\to \Hp$ that maps any group $(G,+)$ to the heap $(G,[-,-,-])$, where the ternary operation $[-,-,-]$ on $G$ is defined by $[a,b,c]=a-b+c$,  for all $a,b,c\in G$. This functor $U$ is faithful, but is not a category equivalence between the categories $\Gp$ and $\Hp$. This occurs for two reasons: on the one hand, a heap can be empty; on the other hand, even if we eliminate the empty heap, the categories $\Gp$ and $\Hp_{\ne\emptyset}$ of non-empty heaps are not equivalent. 

In order to get a category equivalent to the category $\Gp$, we must pass to the category $\Hp_*$ of all {\em pointed heaps}, whose objects are the triplets $(H,[-,-,-],h)$, where $(H,[-,-,-])$ is a heap and $h$ is a fixed element of $H$. Morphisms $$(H,[-,-,-],h)\to (H',[-,-,-],h')$$ in $\Hp_*$ are the heap morphisms 
$$
f\colon (H,[-,-,-])\to(H',[-,-,-])
$$ 
such that $f(h)=h'$. There is a category isomorphism $U_*\colon \Gp\to \Hp_*$ that associates to each group $(G,+)$ with identity (zero element) $0_G$ the pointed heap $(G,[-,-,-],0_G)$ where $[-,-,-]$ is defined by $[a,b,c]=a-b+c$, for every $a,b,c\in G$. The functor $U_*$ is the identity on morphisms. The inverse of $U_*$ is the functor $\Hp_*\to\Gp$ that associates to each pointed heap $(H,[-,-,-],h)$ the group $(H,{\bi}_h)$.

\smallskip

Let us see how this setting generalizes from groups to right groups. 

\begin{lemma}\label{rightheap} 
Let $(B,+,\backslash)$ be a right group, $0$ be an idempotent element of $(B,+)$, and $[-,-,-]$ be the ternary operation on $B$ defined by $[a,b,c]=a+(b\backslash c)$, for all $a,b,c\in B$. Then:
\begin{itemize}
    \item[\rm (1)] The ternary operation $[-,-,-]$ is associative.

\item[\rm (2)] The ternary operation $[-,-,-]$ is {\em left Mal'tsev}, that is, $[x,x,y]=y$ for every $x,y\in B$. 

\item[\rm (3)] The ternary operation $[-,-,-]$ is {\em right weakly Mal'tsev}, that is, $$[x,y,[y,z,w]]=[x,z,w]$$ for every $x,y,z,w\in B$. 

\item[\rm (4)] $[[x,y,z],w,u]=[x,[w,z,y],u]$ for every $x,y,z,w,u\in B$. 

\item[\rm (5)] For every $a,b\in B$, the mapping $[a,b,-]\colon B\to B$ is a bijection. Equivalently, for every $a,b,c\in B$, there is a unique element $x\in B$ such that $[a,b,x]=c$. 

\item[\rm (6)] All the semigroups $(B,{\bi}_y)$, $(y\in B)$, and $(B,+)$ are isomorphic. The semigroup isomorphism $(B,+)\to (B,{\bi}_y)$ is the map $\sigma_0^y\colon (B,+)\to (B,{\bi}_y)$ defined by 
$$
\sigma_0^y(z):=[[z,0,y],z,z]=z+y+(z\backslash z)
$$ 
for every $z\in B$.

\item[\rm (7)] For every $y\in B$, the right group $(B,{\bi}_y)$ contains as subsemigroups: \\
{\rm (a)}  $B\,{\bi}_y\,y=B+e_y$, which is a group with identity $y$ with respect to the operation ${\bi}_y$;\\
{\rm (b)} $\E(B,{\bi}_y)=g_y+\E(B,+)$, which is a right zero semigroup. The right group $(B,{\bi}_y)$ is the direct product of its homomorphic images $B\,{\bi}_y\,y$ and $\E(B,{\bi}_y)$. \\The canonical projections are defined by 
% $[-,y,y]\colon (B,{\bi}_y)\to B\,{\bi}_y\,y$, $b\mapsto b\,{\bi}_y\,y=[b,y,y]=b+e_y$ and $(B,{\bi}_y)\to \E(B,{\bi}_y)=g_y+\E(B,+)$, $b\mapsto g_y+e_b=$``the unique element $x\in B$ such that $[b,y,x]=b$''.
$$[-,y,y]\colon (B,{\bi}_y)\to B\,{\bi}_y\,y,\, b\mapsto b\,{\bi}_y\,y=[b,y,y]=b+e_y
$$ 
and 
$$
(B,{\bi}_y)\to \E(B,{\bi}_y)=g_y+\E(B,+),\, b\mapsto g_y+e_b
$$
where $g_y+e_b =$  ``the unique element $x\in B$ such that $[b,y,x]=b$''.
\end{itemize}
\end{lemma}

\begin{proof} Fix any idempotent element $0\in B$, so that we have a unique decomposition $x=g_x+e_x$, for every element $x\in B$, and the definition $[a,b,c]=a+(b\backslash c)$ can be rewritten as $[a,b,c]=a-g_b+c$, or as $[a,b,c]=g_a-g_b+g_c+e_c$.

(1) We must prove that $[[a,b,c],d,e]=[a,b,[c,d,e]]$, for every $a,b,c,d,e\in B$. The equality $[[a,b,c],d,e]=[a,b,[c,d,e]]$ can be rewritten as $$(a-g_b+c)-g_d+e=a-g_b+(c-g_d+e),$$ which is clearly true for the associativity of the addition $+$.

(2) follows from $[x,x,y]=x-g_x+y=0+y=y$ for every $x, y \in B$.

(3) The right weakly Mal'tsev property is equivalent to $x-g_y+(y-g_z+w)=x-g_z+w$ for all $x,y,z,w \in B$.

(4) follows from $(g_x-g_y+g_z)-g_w+u=(g_x-(g_w-g_z+g_y)+u$ for every $x,y,z,w,u \in B$.

(5) Since $(B,+)$ is a right group, for every $d\in B$ the mapping $\ell_d\colon B\to B$, $x\mapsto d+x$ is a bijection. Equivalently, for every $d,c\in B$, there is a unique element $x\in B$ such that $d+x=c$. Replace $d$ in these conditions with $a-g_b$ getting statement (5).

(6)
Fix an element $y\in B$. Consider the mapping  $\sigma_0^y\colon (B,+)\to (B,{\bi}_y)$, $\sigma_0^y(z)=[[z,0,y],z,z]=z+y+(z\backslash z)=g_z+g_y+e_z$. Then $\sigma_0^y$ is injective, because if $\sigma_0^y(z)=\sigma_0^y(z')$, then $g_z+g_y+e_z=g_{z'}+g_y+e_{z'}$, so that $g_z+g_y=g_{z'}+g_y$ and $e_z=e_{z'}$. Therefore,  $g_z=g_{z'}$ and $e_z=e_{z'}$, which imply $z=z'$. Also, $\sigma_0^y$ is surjective, because if $z\in B$, then $\sigma_0^y(g_z-g_y+e_z)=z$. Finally, $\sigma_0^y$ is a semigroup homomorphism of $(B,+)$ onto $(B,{\bi}_y)$, because for every $z,z'\in B$ one has that $\sigma_0^y(z+z')=g_z+g_{z'}+g_y+e_{z'}$ and $\sigma_0^y(z)\,{\bi}_y\,\sigma_0^y(z')=g_z+g_y+e_z-g_y+g_{z'}+g_y+e_{z'}=g_z+g_{z'}+g_y+e_{z'}$.%=\sigma_0^y(z+z')$.

The proof of (7) consists of routine calculations. 
 \end{proof}

\begin{definition}\label{deff} 
{\rm A {\em right heap} is a pair $(H,[-,-,-])$, where $H$ is a set and $[-,-,-]$ is a ter\-nary operation on $H$ which is associative, left Mal'tsev, and right weakly Mal'tsev. In the category of right heaps, the {\em morphisms} $$f\colon (H,[-,-,-])\to(H',[-,-,-])$$ between two right heaps are the mappings $f\colon H\to H'$ for which $f([x,y,z])=[f(x), f(y), f(z)]$, for all $x,y,z\in H$.}

{\rm A {\em pointed right heap} is a triplet $(H,[-,-,-], h)$, where $(H,[-,-,-])$ is a right heap and $h$ is a fixed element of $H$. Morphisms of pointed right heaps $(H,[-,-,-], h)\to (H',[-,-,-], h')$ are the morphisms of right heaps that map $h$ to $h'$.}

{\rm A {\em pointed right group} is a triplet $(B,+,e)$, where $(B,+)$ is a right group and $e$ is a fixed idempotent element of $B$. Morphisms of pointed right groups $(B,+,e)\to (B',+,e')$ are the semigroup morphisms that map $e$ to $e'$.}
\end{definition}

\begin{Exs}\label{primi}
 (1) If $(G,+)$ is a group, then $(G,[-,-,-])$ is a heap, that is, the operation $[-,-,-]$ is also {\em right Mal'tsev}, i.e., $[x,y,y]=x$ for every $x,y\in G$.

(2) If $(B,+)$ is a right zero semigroup, then $[-,-,-]$ is the third projection $\pi_3\colon B\times B\times B\to B$, that is $
[x,y,z]=z$
for every $x,y,z\in B$. For any set $X$ let $[-,-,-]\colon X\times X\times X\to X$ be the third canonical projection $\pi_3\colon X\times X\times X\to X$. Then $(X,\pi_3)$ is a right heap, which is not a heap if $X$ has at least two distinct elements. We will call these right heaps {\em right zero heaps}. The full subcategory of the category $\Hp$ whose objects are all right zero heaps is clearly isomorphic to the category $\Set$ of sets.
\end{Exs}

\begin{lemma}\label{vip} 
For every $a,b$ in a right heap $(H,[-,-,-])$, the mapping $$[a,b,-]\colon H\to H$$ is a bijection. Its inverse is the mapping $[b,a,-]\colon H\to H$. Therefore, for every $a,b,c\in H$, there is a unique element $x\in H$ such that $[a,b,x]=c$. It is $x=[b,a,c]$.\end{lemma}

\begin{proof} We have that  $[a,b,[b,a,c]]=[a,a,c]=c$ and $[b,a,[a,b,c]]=[b,b,c]=c$, so that the two mappings $[a,b,-],\,[b,a,-]\colon H\to H$ are mutually inverse.\end{proof}

\begin{remark}If, in a heap $H$, we have that $[a,b,c]=d$, then any three of the four elements $a,b,c,d\in H$ determine the fourth. In fact, $[a,b,c]=d$ implies $a=[d,c,b]$, $b=[c,d,a]$ and $c=[ b,a,d]$. Equivalently, for $a,b\in H$, all the mappings $$[-,a,b],[a,-,b],[a,b,-]\colon H\to H$$ are bijections, whose inverses are $[-,b,a], \ [b,-,a]$ and $[b,a,-]$, respectively. For right groups $(B,+)$ and right heaps $(H,[-,-,-])$, we have, by Example~\ref{primi}(2) and Lemma~\ref{vip}, that $[a,b,-]$ is a bijection, while $[-,a,b]$ and $[a,-,b]$ are not bijections in general, because $[g+e,a,b]=[g,a,b]$ and  $[a,g+e,b]=[a,g,b]$, for all $g\in G$ and $e\in E$.\end{remark}

\begin{theorem}\label{xxx'} There is a category isomorphism of 
the category of pointed right groups to the category of pointed right heaps. It associates to every pointed right group 
    $(B,+,e)$ the pointed right heap $(B,[-,-,-],e)$, where $[-,-,-]$ is the ter\-nary operation on $B$ defined by $[a,b,c]=a+(b\backslash c)$, for all $a,b,c\in B$. This category isomorphism is the identity on morphisms.
\end{theorem}

\begin{proof}
    For every right group $(B,+)$,  $(B,[-,-,-])$ is a right heap by Lem\-ma~\ref{rightheap}. If $(B,+)$, $(B',+)$ are right groups and $f\colon B\to B'$ is a semigroup morphism, then $f$ is a right heap morphism. In fact, if $f\colon (B,+)\to (B',+)$ is a semigroup morphism, and $a,b,c,x\in B$, then $b+x=c$ implies $f(b)+f(x)=f(c)$, so that $f(b\backslash c)=f(b)\backslash f(c)$. Thus $f([a,b,c])=f(a+(b\backslash c))=f(a)+f(b\backslash c)=[f(a),f(b),f(c)]$, and $f$ is a right heap morphism. This proves that there is a functor $F$ of the category $\RGrp_*$ of pointed right groups into the category $\RHp_*$ of pointed right heaps.

    Conversely, associate to every pointed right heap $(H,[-,-,-], h)$ the pointed right group $(H, {\bi}_h, h)$. Notice that $h$ is idempotent in this right group. It is now easily seen that there is a functor $G\colon \RHp_*\mapsto \RGrp_*$, $G\colon (H,[-,-,-], h)\mapsto (H, {\bi}_h, h)$.

    In order to prove that $G\circ F$ is the identity functor of $\RGrp_*$, we must prove that for every pointed right group $(B,+,e)$, one has that the operations $+$ and ${\bi}_e$ coincide, that is, that $a+b=[a,e,b]$, for every $a,b\in B$. This follows from the fact that $e$ is idempotent in $(B,+)$.

    To prove that $F\circ G$ is the identity functor of $\RHp_*$, we must prove that for every pointed right heap $(H,[-,-,-], h)$, the operations $[-,-,-]$ and $[-,-,-]'$ coincide, where, for every $a,b,c\in H$, $[a,b,c]'=a{\bi}_h x$ and $x$ is the unique element of $H$ such that $b{\bi}_h x=c$. That is, we must prove that, for every $a,b,c,x\in H$, if $[b,h,x]=c$, then $[a,b,c]=[a,h,x]$. In fact, $[b,h,x]=c$ implies $[a,b,c]=[a,b,[b,h,x]]=[a,h,x]$.
\end{proof}

\begin{theorem}\label{xxx} There is a product-preserving, faithful, essentially surjective functor from 
the category of right groups to the category of non-empty right heaps. It associates to every right group 
    $(B,+)$ the right heap $(B,[-,-,-])$, where $[-,-,-]$ is the ternary operation on $B$ defined by $[a,b,c]=a+(b\backslash c)$, for all $a,b,c\in B$. The functor is the identity on morphisms.
\end{theorem}
\begin{proof}
    Most properties follow from Theorem~\ref{xxx'} and its proof. In that proof we had already seen that if $(B,+)$, $(B',+)$ are right groups and $f\colon B\to B'$ is a semigroup morphism, then $f$ is a right heap morphism. Conversely, if $(B,+)$, $(B',+)$ are right groups and $f\colon (B,[-,-,-])\to (B',[-,-,-])$ is a right heap morphism, fix an idempotent element $e$ in $(B,+)$. Then $[f(e),f(e),f(e)]=f(e)$, so that $f(e)$ is an idempotent element in the magma $(B',{\bi}_{f(e)})$. By Lemma~\ref{rightheap}(6),
    the right group $(B',+)$ is isomorphic to the magma $(B',{\bi}_{f(e)})$ via $\sigma_0^{f(e)}\colon (B',+)\to (B',{\bi}_{f(e)})$ defined by $\sigma_0^{f(e)}(z)=[[z,0,{f(e)}],z,z]=z+{f(e)}+(z\backslash z)=z+(z\backslash z)=z$, for every $z\in B'$. Hence, the operations $+$ and ${\bi}_{f(e)}$ on $B'$ coincide. Therefore, $f(a+b)=f(a+(e\backslash b))=f([a,e,b])=[f(a),f(e),f(b)]=f(a){\bi}_{f(e)}f(b)=f(a)+f(b)$. Thus, $f$ is a semigroup morphism.
 In particular, the functor is faithful.

As far as essential surjectivity is concerned, let $(H,[-,-,-])$
be a non-empty right heap. Fix an element $y\in H$. It is easily seen that
$(H,{\bi}_y)$ is a right group (see Lemma~\ref{vip}). Define another ternary operation $[-,-,-]'$ on the set $H$ setting $$[a,b,c]'=a\,{\bi}_y\,(b\,{\backslash}_y\,c),$$ where $b\,{\backslash }_y\,c$ is the unique element $x$ of $B$ such that $b\,{\bi}_y\,x=c$. That is, $b\,{\backslash}_y\,c$ is the unique element $x$ of $B$ such that $[b,y,x]=c$. We must show that the two operations $[-,-,-]$ and $[-,-,-]'$ concide. Now, for every $a,b,c\in H$, we have that $[a,b,c]'=a\,{\bi}_y\,x=[a,y,x]$ and $[b,y,x]=c$. Thus $[a,b,c]=[a,b,[b,y,x]]=[a,y,x]=[a,b,c]'$, that is, 
two operations $[-,-,-]$ and $[-,-,-]'$ concide. This proves 
essential surjectivity.

Finally, it is clear that the product in the category of right groups and in the category of non-empty right heaps is the cartesian product with the component-wise operation. Therefore, the functor is product-preserving.
\end{proof}

The functor of Theorem~\ref{xxx} is not a category equivalence: the categories $\RGrp$ and $\RHp_{\ne\emptyset}$ of non-empty right heaps are not equivalent. In order to see it, assume by contradiction that $\RGrp$ and $\RHp_{\ne\emptyset}$ be equivalent. Both categories have terminal objects, the trivial groups and the trivial heaps (those of cardinality one). Thus, trivial groups must correspond to trivial heaps in the category equivalence. Now, if $G$ is any fixed non-trivial group, then its corresponding (right) heap $H$ must have more than one element. But there is only one morphism in $\Gp$ from a trivial group to $G$, while the set of all morphisms in $\RHp_{\ne\emptyset}$ from a trivial right heap to $H$ has the same cardinality as $H$, hence has cardinality $>1$. Therefore, there do not exist category equivalences from  $\RGrp$ to $\RHp_{\ne\emptyset}$.

\smallskip

In view of Theorem~\ref{xxx'}, we have the following Metatheorem: Let $p,q$ be two terms in the language of heaps (equivalently, in the language of right heaps). Assume that the identity $p=q$ holds for all heaps. Then the identity $p=q$ holds for all right heaps, provided that the last variable that appears in the term $p$ is equal to the last variable that appears in the term $q$.

For instance, it is well known that the identity $[[x,y,z],w,u]=[x,[w,z,y],u]$ holds for all heaps. Notice that in the two terms $[[x,y,z],w,u]$ and $[x,[w,z,y],u]$, the last variable is $u$ for both terms. Now if $H$ is a right heap, we can fix an element $0$ in $H$, so that the pointed right heap $(H,[-,-,-],0)$ corresponds to a pointed right group via Theorem~\ref{xxx'}, the pointed right group $(B,+,0)$, say. Then, in the pointed right heap $(H,[-,-,-],0)$ the identity $[[x,y,z],w,u]=[x,[w,z,y],u]$ with $x,y,z,w,u\in H$, can be rewritten as $[[g_x,g_y,g_z],g_w,g_u]+e_u=[g_x,[g_w,g_z,g_y],g_u]+e_u$. But the identity $[[g_x,g_y,g_z],g_w,g_u]=[g_x,[g_w,g_z,g_y],g_u]$ holds, because it holds in all heaps, and therefore it also holds in the pointed heap corresponding to the pointed right group $(B+0,+,0)$. We have thus proved %that:
the following result.

\begin{theorem} For every $x,y,z,w,u$ in a right heap, we have that $$[[x,y,z],w,u]=[x,[w,z,y],u].$$\end{theorem}

We will come back to our Metatheorem in Remark~\ref{Meta}(a).

\begin{proposition}\label{functor}
    Let $(A,a_0)$ be a pointed set. Then there is a bijection between the set $\{\,+\,\mid (A,+)$ is a right group and $a_0$ is idempotent in $(A,+)\,\}$ of all binary operations $+$ on $A$ such that $(A,+)$ turns out to be a right group in which $a_0$ is idempotent, and the set $\{\,[-,-,-]\,\mid (A,[-,-,-])$ is a right heap$\,\}$ of all ternary operations $[-,-,-]$ on $A$ such that $(A,[-,-,-])$ turns out to be a right heap. If $+$ and $[-,-,-]$ correspond in this bijection, then:
\begin{itemize}
    \item[\rm (1)] An equivalence relation on $A$ is a congruence for the right group $(A,+)$, that is, is compatible with both the operation $+$ and its right inverse $\backslash$, if and only if it is compatible with the ternary operation $[-,-,-]$.

    \item[\rm (2)] A subset of $A$ containing $a_0$ is closed for both the operation $+$ and its right inverse $\backslash$ if and only if it  is closed for the operation $[-,-,-]$.
    
      \item[\rm (3)] A subset $A'$ of $A$ containing $a_0$ is closed for the operation $+$ and is a right group if and only if $A'$ is a right subheap of the right heap  $(A,[-,-,-])$, that is, for every $a',b'\in A'$, the bijection $[a',b',-,]\colon A\to A$ restricts to a bijection $[a',b',-,]\colon A'\to A'$.
  
    \item[\rm (4)] Let $f\colon A\to A$ be an idempotent mapping. If $f$ is a semigroup endomorphism of the semigroup $(A,+)$, then $f$ is a right heap endomorphism of $(A,[-,-,-])$. Conversely, if $f$ is a right heap endomorphism of $(A,[-,-,-])$ and $f(a_0)=a_0$, then $f$ is a right group endomorphism of the right group $(A,+)$.
    \end{itemize}\end{proposition}

The proof of Proposition~\ref{functor} is an elementary exercise essentially based on Lemma~\ref{rightheap} and is left to the reader. The ternary operation $[-,-,-]$ associated to the binary operation $+$ is defined by $[a,b,c]=a+(b\backslash c)$, for all $a,b,c\in B$. Conversely, the binary operation associated to the ternary operation $[-,-,-]$ is the operation $+_{a_0}$. 

\smallskip

We know that if $e\in B$ is an idempotent element of the right group $(B,+)$, there is a corresponding direct-product decomposition $B=(B+e)\times E$, where the two projections are $B\to B+e$, $b\mapsto b+e$, and $B\to E$, $b\mapsto \,$``the unique element $x\in B$ such that $b+x=b$''. These two projections can be also viewed as the idempotent endomorphisms $\sigma_e,\tau$ of the right group $(B,+)$. 
Applying the functor of Theorem~\ref{xxx}, we get that the right heap $(B,[-,-,-])$ decomposes as a direct product of a heap $(B+e, [-,-,-])$ and a right heap $(E,\pi_3)$. Here the heap $(B+e, [-,-,-])$ is isomorphic to the quotient right heap $B/{\sim}$, where $\sim$ is the congruence on the right heap $(B,[-,-,-])$ generated by the set $\{\,([x,y,y],x)\mid x,y\in B\,\}$ (so that the operation induced by the ternary operation of $B$ on the quotient $B/{\sim}$ is also right Mal'tsev). Similarly, the right heap $(E,\pi_3)$  is isomorphic to the quotient heap $B/{\equiv}$, where $\equiv$ is the congruence on the right heap $(B,[-,-,-])$ generated by the set $\{\,([x,y,z],z)\mid x,y,z\in B\,\}$. The two congruences $\sim$ and $\equiv$ permute, and their product ${\sim}\circ{\equiv}$ and ${\equiv}\circ{\sim}$ coincide with the trivial congruence $\omega$ on $B$. 
%In fact, for every $x,y\in B$, we have that $(x,y)=(x,[x,y,y])\circ([x,y,y],y)\in {\sim}\circ{\equiv}$, so that, ${\sim}\circ{\equiv}=\omega$, and $(x,y)=(x,[y,x,x])\circ([y,x,x],y)\in 
%{\equiv}\circ{\sim}$, so ${\equiv}\circ{\sim}=\omega$. 

In right heaps we have %:
the following result.

\begin{proposition}\label{dec}
    Fix an element $a$ in a right heap $H$. Then:
\begin{itemize}
\item[\rm (1)] The mapping $p_a\colon H\to H$, defined by $p_a(x)=[x,a,a]$, for every $x\in H$, is an idempotent endomorphism of the right heap $H$. The image $G_a$ of $p_a$ is a right subheap of $H$ and is a heap.

\item[\rm (2)] The mapping $q_a\colon H\to H$, defined by $q_a(x)=[a,x,x]$, for every $x\in H$, is an idempotent endomorphism of the right heap $H$. The image $E_a$ of $q_a$ is a right subheap of $H$ and is a right zero heap.

    \item[\rm (3)] $x=[p_a(x),a,q_a(x)]$, for every $x\in H$. 

    \item[\rm (4)] $H$ is isomorphic to the direct product of the heap $G_a$ and the right zero heap $E_a$. The isomorphism is $\nu_a\colon H\to G_a\times E_a$, $x\in H\mapsto ([x,a,a], [a,x,x])$. Its inverse is the mapping $\mu_a\colon G_a\times E_a\to H$, $\mu_a\colon (y,z)\in G_a\times E_a\mapsto [y,a,z]$.

     \item[\rm (5)] $G_a\cap E_a =\{a\}$.
     
      \item[\rm (6)] $x=[p_a(x),a,q_a(x)]$ is the unique way of writing $x$ in the form $x=[y,a,z]$ with $y$ in the heap $G_a$ and $z$ in the right zero heap $E_a$.

      \item[\rm (7)] The two idempotent endomorphisms $p_a,q_a\colon H\to H$ are mutually orthogonal, in the sense that $p_a\circ q_a=q_a\circ p_a$ is the constant morphism $H\to H$ that sends all elements of $H$ to $a$.

\item[\rm (8)]  The kernels of $p_a$ and $q_a$ are the congruences $\sim$ and $\equiv$, respectively.
\end{itemize}\end{proposition} 

\begin{proof}
(1) The mapping $p_a$ is idempotent because $[[x,a,a],a,a]=[x,a,a]$ The mapping $p_a$ is a right heap homomorphism, because $[[x,a,a],[y,a,a],[z,a,a]]=[[x,y,z],a,a]$ (this is an identity that holds for heaps, and the last variable in both terms is $a$.) The image of a right heap homomorphism is always a right subheap. Finally, the ternary operation is right Mal'tsev on $G_a=p_a(B)$, because if $[x,a,a],[y,a,a]\in p_a(B)$, then $[[x,a,a],[y,a,a],[y,a,a]]=[[x,y,y],a,a]=[x,a,a]$.

(2) The mapping $q_a$ is idempotent because, for every $x \in H$, $$[a,[a,x,x],[a,x,x]]=[[a,x,x],[a,x,x],[a,x,x]]=[a,x,x].$$ 
The ternary operation on the image $E_a$ of $q_a$ is the third canonical projection because, for every $x,y,z \in H$,
\begin{equation} [[a,x,x],[a,y,y],[a,z,z]]=[a,a,[a,z,z]]=[a,z,z],
\label{parziale}
\end{equation} 
hence $E_a$ is a right zero heap. Moreover,  %\begin{equation} 
\begin{align}\label{parziale2}[a,[x,y,z],[x,y,z]]&=[[a,z,y],x,[x,y,z]]=[[[a,z,y],x,x],y,z]= \\ \notag &=[[a,z,y],y,z]=[a,z,[y,y,z]]=[a,z,z].
\end{align} 
From (\ref{parziale}) and (\ref{parziale2}) we get that $q_a$ is a right heap homomorphism.

(3) For all $x\in H$, 
$$
[p_a(x),a,q_a(x)]=[[x,a,a],a,[a,x,x]]=[x,a,[a,x,x]]=[x,x,x]=x.
$$

(4) The mapping $H\to G_a\times E_a$, $x\in H\mapsto ([x,a,a], [a,x,x])$, is a right heap homomorphism by (1) and (2). In order to prove that the mapping $\mu_a\colon G_a\times E_a\to H$, $\mu_a\colon (y,z)\in G_a\times E_a\mapsto [y,a,z]$, is a right heap homomorphism, it is convenient to first prove an auxiliary lemma that shows, for right heaps, the fact that idempotents in a right group are left identities of the group.

\begin{lemma} Let $a,b,c,d$ be elements in a right heap $H$, and $y$ be an element in its right zero subheap $E_a$. Then:
\begin{itemize}
    \item[\rm (i)] $[[b,a,y],c,d]=[b,c,d]$.

\item[\rm (ii)] $[c,[b,a,y],d]=[c,b,d]$.
\end{itemize}
\end{lemma}
\begin{proof} If $y\in E_a$, which is the image of the idempotent mapping $q_a$, we have that $y=q_a(y)$. Therefore $y=[a,y,y]$ and so:
\begin{itemize}
    \item[\rm (i)] $[[b,a,y],c,d]=[[b,a,[a,y,y]],c,d]=[[b,y,y],c,d]=[b,c,d]$.
    \item[\rm (ii)] $[c,[b,a,y],d]=[c,[b,a,[a,y,y]],d]=[c,[b,y,y],d]=
    [c,b,d]$.
\end{itemize}
\end{proof}

We are now ready to prove that $\mu_a$ is a right heap homomorphism. For every $(x,y),(x',y'),(x'',y'')\in G_a\times E_a$, we have that $$\begin{array}{l}\mu_a([(x,y),(x',y'),(x'',y'')])=\mu_a([x,x',x''],[y,y',y''])= \\ \qquad\qquad =[[x,x',x''],a,[y,y',y'']]=[[x,x',x''],a,y'']=[x,x',[x'',a,y'']]= \\ \qquad\qquad =[[x,a,y],x',[x'',a,y'']]=[[x,a,y],[x',a,y'],[x'',a,y'']]= \\ \qquad\qquad =[\mu_a(x,y),\mu_a(x',y'),\mu_a(x'',y'')].\end{array}$$ This proves that $\mu_a$ is a right heap homomorphism. 

In view of (3), to prove that the two morphisms are mutually inverse, it suffices to show that $p_a([x,a,y])=x$ and $q_a([x,a,y])=y$, for every~$x\in G_a$ and~$y\in E_a$, since $\mu_a\nu_a = \id_H$ is trivially satisfied. 
Now, $p_a([x,a,y])=[[x,a,y],a,a]=[x,a,[y,a,a]]=[x,a,a]$ because %$y,a\in H_a$ 
$y,a\in E_a$ and the ternary operation is the third canonical projection on $E_a$. Therefore, $p_a([x,a,y])=[x,a,a]=x$, as desired.

As far as $q_a$ is concerned, since $y\in E_a$,
\begin{align*}
    q_a([x,a,y])
    =[a,[x,a,y],[x,a,y]]
    = [a,x,[x,a,y]]
    = [a,a,y]
    = y
\end{align*}
by Lemma \ref{rightheap}((2) and (3)).

(5) From $a=[a,a,a]$ it follows that $a\in G_a\cap E_a$. 

Conversely, suppose that $x\in G_a\cap E_a$. Since $p_a$ is idempotent, $p_a(x) = x$.  Since $E_a$ is a right zero heap and $a\in E_a$,  $p_a(x) = [x,a,a] = a$. Hence $x=a$. 

(6) In view of (3), we must show that if $x=[y,a,z]$ with $y\in G_a$ and $z\in E_a$, then $y=p_a(x)$ and $z=q_a(x)$. Assume $x=[y,a,z]$ with $y\in G_a$ and $z\in E_a$. We find that $$p_a(x)=p_a([y,a,z])=[p_a(y),p_a(a),p_a(z)]=[y,a,[z,a,a]]=[y,a,a]$$ because $z,a\in E_a$ and on $E_a$ the ternary operation is the third canonical projection. 
Hence, $p_a(x)=y$. For the other equality, $q_a(x)=q_a([y,a,z])=$\linebreak $=[q_a(y),q_a(a),q_a(z)]=[[a,y,y],a,z]=[a,a,z]$. Since $a,z\in E_a$ and on $E_a$ the ternary operation is the third projection, it follows that 
$q_a(x)=[a,a,z]=z$, as desired.

(7) Let $x\in H$. Then
     $(p_a\circ q_a)(x)=p_a(q_a(x))=[[a,x,x],a,a]=[a,a,a]=a$ and  \begin{align*}(q_a\circ p_a)(x)&=q_a(p_a(x))=[a,[x,a,a],[x,a,a]]=\\ &=[[a,a,a],x,[x,a,a]]=[a,x,[x,a,a]]=[a,a,a]=a.\end{align*}

(8) The kernel $\theta_a$ of the idempotent mapping $p_a$ is clearly the congruence generated by the set $\{\,([x,a,a],x)\mid x\in H\,\}$. Therefore, $\sim$ is the join of the set $\{\,\theta_a\mid a\in H\,\}$ in the lattice of all congruences of the right heap $(H,[-,-,-])$. 
In particular, $\theta_a$ is contained in $\sim$. Let us prove that, for every $a,b\in H$, the kernel of the two congruences $\theta_a$ and $\theta_b$ coincide. That is, let us show that, for every $x,y\in H$, $[x,a,a]=[y,a,a]$ if and only if $[x,b,b]=[y,b,b]$. If $[x,a,a]=[y,a,a]$, then $[[x,a,a],b,b]=[[y,a,a],b,b]$, so that $[x,[b,a,a],b]]=[y,[b,a,a],b]]$. Therefore, $[x,b,b]=[y,b,b]$. Similarly, for $[x,b,b]=[y,b,b]$ implies $[x,a,a]=[y,a,a]$. Thus $\theta_a=\theta_b$.
 
 As far as the kernel $\nu_a$ of the idempotent mapping $q_a$ is concerned,  we have that the the kernel $\nu_a$ of $q_a$ is the congruence generated by the set $\{\,([a,a,x],x)\mid x\in H\,\}$, and $\equiv$ is the congruence generated by the set $\{\,([x,y,z],z)\mid x,y,z\in H\,\}$. In particular, $\nu_a$ is contained in $\equiv$. For the converse, we must show that $ [x,y,z]\nu_a z$, for every $x,y,z\in H$. This is equivalent to proving that $q_a( [x,y,z])=q_a( z)$. Now, $q_a( [x,y,z])=[a,[x,y,z],[x,y,z]]=[a,z,[y,x,[x,y,z]]]=[a,z,z]=q_a(z)$. Therefore,  $\nu_a$ and $\equiv$ coincide.
\end{proof}

Clearly, the mappings $p_a$ and $q_a$ are the analogues of the mappings $\sigma_e$ and $\tau$ for right groups (Subsection~\ref{1.4}).

\begin{remarks}\label{Meta} (a) As a consequence of the fact that every right heap is the direct product of a heap and a right zero heap (Proposition~\ref{dec}), we get that an identity $p=q$ in the language of heaps holds for all right heaps if and only if it holds for all heaps and for all right zero heaps. 
Now, the identity $p=q$ holds for all right zero heaps if and only if the last variable that appears in the term $p$ is equal to the last variable that appears in the term $q$. This gives another interpretation of our previous Metatheorem.

(b) Proposition~\ref{dec}((6) and (7)) implies that $\mu(p_a\times c_a\times q_a)\Delta=\id_H$, where $H$ is any right heap, $\Delta\colon H\to H\times H\times H$ is the mapping $\Delta(x)=(x,x,x)$, $c_a\colon H\to H$ is the constant homomorphism constantly equal to $a$, and $\mu\colon H\times H\times H\to H$ is the mapping $\mu(x,y,z)=[x,y,z]$.

(c) From Proposition~\ref{dec}(8) we get that, for a right heap $H$ and any $a,b\in H$, we have that $G_a\cong G_b$ and $E_a\cong E_b$. 

(d) The identity $[a,b,c]=g_a-g_b+g_c+e_c$, for all $a,b,c\in B$, already noticed in the first paragraph of the proof of Lemma~\ref{rightheap} shows the direct-product decomposition of the right heap $(B,[-,-,-])$ with respect to the element $0$ very clearly: one has $B=[B,0,0]\times E$, where the ternary operation on $[B,0,0]$ is $[g_a,g_b,g_c]=g_a-g_b+g_c$, and the ternary operation on $E$ is $[e_a,e_b,e_c]=e_c$. 
\end{remarks}

For every $a\in H$, the decomposition $H\cong G_a\times E_a$ is a product decomposition in the category $\Hp$, but is not a coproduct decomposition in $\Hp$. It is a coproduct decomposition in the category $\Hp_*$, i.e.,  $(H,a)$ is the coproduct of $(G_a,a)$ and $(E_a,a)$ in $\Hp_*$. Notice the difference with the case of right groups: for a right group $S$, for every idempotent element $e\in E$ there is a product decomposition $Se\times E$, where the zero right group $E$ is fixed, while its complement, the group $Se$ depends on the idempotent $e$; for a right heap $H$, for every element $a\in H$ there is a product decomposition $G_e\times E_e$, where both the heap $G_e$ and the zero right heap $E_a$ depend on the fixed element $a$.

\begin{remark}\label{si}
    Lemma \ref{rightheap}(5) explains why in the study of right groups, the element $0$ of $B$ such that $B=(B+0)\times E$ can be chosen as any fixed {\em idempotent} element of $B$. In fact,  $0$ can be chosen as {\em any} element $y$ of $B$, but in this case we must deal with the operation ${\bi}_y$ defined by $x\,{\bi}_y\,z=g_x-g_y+g_z+e_z$. If $y$ is idempotent, then $g_y=0$, so that $x\,{\bi}_y\,z=g_x-0+g_z+e_z=g_x+g_z+e_z=x+z$, that is, ${\bi}_y$ is the original operation $+$ on the right group $(B,+)$.

    The functor described in Theorem~\ref{xxx} is not a category isomorphism, because, for a fixed right group $(B,+)$, $(B,+)$ is mapped to the right heap $(B,[-,-,-])$, but also all right groups $(B,{\bi}_y)$ ($y\in B$) are mapped to the same right heap $(B,[-,-,-])$. Also, all right groups $(B,{\bi}_y)$ ($y\in B$)  are isomorphic, but they can be distinct: it suffices to take as $y$ an element of $B$ that is not idempotent in $(B,+)$, and then $(B,+_y)$ and $(B,+)$ are distinct.

    Also notice that, trivially, the element $y\in B$ is always idempotent in $(B,{\bi}_y)$, because $y\,{\bi}_y\,y=y+(-y+y)=g_y+e_y=y$.
\end{remark}

It is now clear that our left $RG$-semibraces $(B,+,\circ)$ can be equivalently described as:
\begin{definition}
    A {\em left $RG$-semitruss} is a triplet $(T,[-,-,-],\circ)$, where\linebreak $(T,[-,-,-])$ is a right heap, $(T,\circ)$ is a right group, and left distributivity
\begin{equation*} 
a\circ [b, c,d] = [a\circ b,a\circ c, a\circ d]
\end{equation*} 
holds, for every $a,b,c\in T$.
\end{definition}

\begin{theorem}\label{2.14} 
{\rm (a)} If $(B,+,\circ)$  is a left $RG$-semibrace, then $(B,[-,-,-],\circ)$ is a left $RG$-semitruss, where $[-,-,-]: B \times B  \times B \to B$ is the ternary operation defined by $[a,b,c]=a+(b\backslash c)$, for all $a,b,c \in B$.

{\rm (b)}  If $(T,[-,-,-],\circ)$ is a left $RG$-semitruss and $y\in T$, then $(T,{\bi}_y,\circ)$ is a left $RG$-semibrace.

{\rm (c)} There is a faithful, essentially surjective functor from 
the category of left $RG$-semibraces to the category of left $RG$-semitrusses. \end{theorem}

\begin{proof} (a) Let  $(B,+,\circ)$  be a left $RG$-semibrace, $\backslash$ be the left inverse operation of the addition, and $[a,b,c]=a+(b\backslash c)$. In order to prove that $(B,[-,-,-],\circ)$ is a left $RG$-semitruss, we must show that $a\circ [b,c,d]=[a\circ b,a\circ c,a\circ d]$ for every $a,b,c,d\in B$. 
Now, $a\circ [b,c,d]=a\circ (b+(c\backslash d))=a\circ (b+x)$, where $x\in B$ is the unique element such that $c+x=d\label{aaa}$. Thus $a\circ [b,c,d]=a\circ (b+x)=a\circ b +(a\backslash a\circ x)=a\circ b +y\label{fff}$, where  $y\in B$ is the unique element such that $a+y=a\circ x\label{ccc}$. Also, we have $a\circ d=a\circ (c+x)=a\circ c+(a\backslash(a\circ x))=a\circ c+y\label{bbb}$. Therefore, $[a\circ b,a\circ c,a\circ d]=a\circ b+((a\circ c)\backslash(a\circ d))=a\circ b+y$. Hence, $a\circ [b,c,d]=[a\circ b,a\circ c,a\circ d]$.

{\rm (b)}  Let $(T,[-,-,-],\circ)$ be a left $RG$-semitruss. Fix an element $y\in T$. We must prove that $a\circ(b\,{\bi}_y\,c)=(a\circ b){\bi}_y(a\backslash_y(a\circ c))$, where $a\backslash_y -$ is the left inverse mapping of $a\,{\bi}_y-=[a,y,-]$. Therefore, $a\backslash_y -=[y,a,-]$. Hence, $$a\circ(b\,{\bi}_yc)=a\circ[b,y,c]=[a\circ b,a\circ y,a\circ c],$$ and
$(a\circ b){\bi}_y((a\circ y)\backslash_y(a\circ c))=[a\circ b,y,[y,a\circ y,a\circ c]]=[a\circ b,a\circ y,a\circ c]$, which allows us to conclude.

{\rm (c)} is now clear.
\end{proof}

\section{Left quasiheaps}\label{4}

In this section, we will see how part of the previous theory developed for right groups generalizes to left quasigroups.

\begin{lemma}\label{rightheap'} 
Let $(Q,\cdot,\backslash)$ be a left quasigroup, and let $[-,-,-]$ be the ternary operation on $Q$ defined by $[a,b,c]=a\cdot(b\backslash c)$ for all $a,b,c\in Q$. Then:

{\rm (1)} The ternary operation $[-,-,-]$ is left Mal'tsev. 

{\rm (2)} The ternary operation $[-,-,-]$ is right weakly Mal'tsev. 

{\rm (3)}
For every $a,b\in Q$, the mapping $[a,b,-]\colon Q\to Q$ is a bijection. Equivalently, for every $a,b,c\in Q$, there is a unique element $x\in Q$ such that $[a,b,x]=c$. 

{\rm (4)} If $y\in Q$ is a left identity in $(Q,\cdot)$, then the operations  $\cdot_y$ and $\cdot$ on $Q$ coincide. 

{\rm (5)} For every $y\in Q$, the magma $(Q,\cdot_y)$ is a left quasigroup in which $y$ is a left identity.
\end{lemma}

\begin{proof} Notice that the ternary operation is defined by  $[a,b,c]=\ell_a(\ell_b)^{-1}(c)$,  for every $a,b,c\in Q$. 

(1) $[x,x,y]=\ell_x(\ell_x)^{-1}(y)=y$. 

(2) $[x,y,[y,z,w]]=\ell_x(\ell_y)^{-1}(\ell_y(\ell_z)^{-1}(w))=\ell_x(\ell_z)^{-1}(w)=[x,z,w]$.

(3) The mapping $[a,b,-]\colon Q\to Q$ is the composite mapping $\ell_a(\ell_b)^{-1}$, i.e., it is a composite mapping of two bijections.

(4)
If $y$ is a left identity of $(Q,\cdot)$, then $$x\,\cdot_y\,z=[x,y,z]=\ell_x(\ell_y)^{-1}(z)=\ell_x(\id_Q)^{-1}(z)=\ell_x(z)=x\cdot z.$$ 

(5) The magma $(Q,\cdot_y)$ is a left quasigroup by (3), and $y$ is a left identity in $(Q,\cdot_y)$ by (1).
\end{proof}

A {\em left quasiheap} is a set $H$ with a ternary operation $[-,-,-]$ such that:

{\rm (1)} $[-,-,-]$ is {\em left Mal'tsev}, that is $[x,x,y]=y$. 

{\rm (2)}  $[-,-,-]$ is {\em right weakly Mal'tsev}.

\begin{lemma} For a left quasiheap $H$, the mapping $[a,b,-]\colon H\to H$ is a bijection, for every $a,b\in H$.
    \end{lemma}

    \begin{proof} Since $[-,-,-]$ is left Mal'tsev and right weakly Mal'tsev, we know that $$[a,b,[b,a,x]]=[a,a,x]=x,$$ so that $[b,a,-]\colon H\to H$ is a right inverse for $[a,b,-]\colon H\to H$. Exchanging $a$ and $b$ we see that $[b,a,-]$ is also a left inverse for $[a,b,-]$. Therefore $[a,b,-]$ and $[b,a,-]$ are mutually inverse mappings $H\to H$.
        \end{proof}

         Clearly, right heaps are exactly the associative left quasiheaps.

      \begin{theorem}
          There is a product-preserving, faithful, essentially surjective functor from 
the category of left quasigroups to the category of left quasiheaps. It associates to every left quasigroup 
    $(Q,\cdot,\backslash)$ the left quasiheap $(Q,[-,-,-])$, where $[-,-,-]$ is the ternary operation on $Q$ defined by $[a,b,c]=a\cdot(b\backslash c)$ for every $a,b,c\in Q$. The functor is the identity on morphisms.
      \end{theorem}
      \begin{proof} The functor is well defined by Lemma~\ref{rightheap'}((1) and (2)). Notice that if a mapping $f$ is a left quasigroup morphism, it preserves $\cdot$ and $\backslash$, hence it preserves $[-,-,-]$. Thus the functor is clearly product-preserving and faithful.

      As far as essential surjectivity is concerned,  if the left quasiheap is empty, it is the image of the empty left quasigroup. Hence, let $(H,[-,-,-])$
be a non-empty left quasiheap. Fix an element $y\in H$. It is easily seen that
$(H,{\bi}_y)$ is a left quasigroup with $y$ as a left identity. The right inverse $\backslash_y$ of the operation $\cdot_y$ is defined in such a way that $a\backslash_y-$ is the inverse of the bijection $a\cdot_y-=[a,y,-]$, which is $[y,a,-]$. Therefore, $a\backslash_yb=[y,a,b]$ for every $a,b\in H$. Define another ternary operation $[-,-,-]'$ on the set $H$ setting $[a,b,c]'=a\,\cdot_y\,(c\,\backslash_y\,c)$ for all $a,b,c\in H$. We must show that the two operations $[-,-,-]$ and $[-,-,-]'$ coincide. But, for every $a,b,c\in H$, we have that $[a,b,c]'=a\cdot_y(b\backslash_y\,c)=[a,y,[y,b,c]]=[a,b,c]$. This proves 
essential surjectivity.
      \end{proof}

    The involutory category isomorphism $\LQGrp\to \LQGrp$ that sends every left quasigroup $(Q, \cdot,\backslash)$ to the left quasigroup $(Q, \backslash, \cdot)$ and is the identity on left quasigroup morphisms has the following interpretation in left quasiheaps. For every left quasiheap $(H,[-,-,-])$ and every fixed element $y\in H$, define another ternary operation $[-,-,-]'_y$ on $H$  setting $[a,b,c]'_y=[y,a,[b,y,c]]$ for all $a,b,c\in H$. 
    Then the following are satisfied: \\
(1) $(H,[-,-,-]'_y)$ is a  left quasiheap, which we call the {\em inverse of $(H,[-,-,-])$ with respect to $y$}.
%. We call $(H,[-,-,-]'_y)$ the {\em inverse of $(H,[-,-,-])$ with respect to $y$}.

(2) If $(H,[-,-,-]),(L,[-,-,-])$ are left quasiheaps and $f\colon H\to L$ is a mapping, then $f\colon (H,[-,-,-])\to (L,[-,-,-])$ is a left quasiheap morphism if and only if $f\colon (H,[-,-,-]'_y)\to (L,[-,-,-]'_{f(y)})$ is a left quasiheap morphism.

(3) The construction is involutory, that is $[[-,-,-]'_y]'_y=[-,-,-]$.

\noindent The proofs are elementary.

\bibliographystyle{amsalpha}

\begin{thebibliography}{A}

\bibitem [AFMS]{AFMS} A. Albano, A. Facchini, M. Mazzotta, and P. Stefanelli, \textit{Right groups and the set-theoretic Yang-Baxter equation.} Submitted for publication (2026). Also available at https://arxiv.org/abs/2605.25660

\bibitem [B]{Brze} T. Brzezi\'nski, \textit{Trusses: between braces and rings.} Trans. Amer. Math. Soc.
\textbf{372} (2019), 4149--4176.

\bibitem [CP]{ClPr61} A. H. Clifford and G. B. Preston, \textit{The algebraic theory of semigroups. Vol. I},
Amer. Math. Soc., Providence, RI, 1961.


%\bibitem[D]{Dr92} V. G. Drinfel'd, \textit{On some unsolved problems in quantum group theory}, Quantum groups ({L}eningrad, 1990), Lecture Notes in Math., 1510, 1--8.

\bibitem [E]{Australian} C. C. Edmunds, \textit{Interchange rings.} J. Aust. Math. Soc.
\textbf{101} (2016), 310--334.

\bibitem [ESS]{ESS99} Etingof, P. and Schedler, T. and Soloviev, A., \textit{Set-theoretical solutions to the quantum {Y}ang-{B}axter equation} Duke Math. J.
\textbf{100}(2) (1999), 169--209.


%\bibitem [F]{ArticoloperLens} A. Facchini,  \textit{Trusses, weak trusses, ditrusses.} Submitted for publication (2026). Also available at http://arxiv.org/abs/2510.23185

%\bibitem [FF]{FacFin} A. Facchini,  and C. A. Finocchiaro,  \textit{A pretorsion theory for right groups.} Submitted for publication (2026). Also available at https://arxiv.org/abs/2603.23982

\bibitem [GV]{GuVe17} L. Guarnieri,  and L. Vendramin,  \textit{Skew braces and the Yang-Baxter equation.} Math. Comp.
\textbf{86} (2017), 2519--2534.

\bibitem [HL]{HL} C. D. Hollings, and M. V. Lawson, \textit{Wagner's theory of generalised heaps},
Springer, Cham, 2017.

\bibitem [JMZ]{JePiZa19} P. Jedli\v cka, A. Pilitowska, and A. 
              Zamojska-Dzienio, \textit{The retraction relation for biracks.} J. Pure Appl. Algebra \textbf{223} (2019), 3594--3610.

\bibitem [K]{Koch5} J. Koch, \textit{Note on commutativity of double semigroups and two-fold monoidal categories.} J. Homotopy Relat. Struct.
\textbf{2} (2007), 217--228.

\bibitem [Sta]{St06} D. Stanovsk\'y, \textit{On axioms of biquandles.} J. Knot Theory Ramifications
\textbf{15} (2006), 931--933.

\bibitem [Ste]{St25} P. Stefanelli, \textit{Dual weak braces and {P}\l onka sums of solutions of the
              Yang-Baxter equation},
Loops'23: Nonassociative algebra, vol. 129,
Banach Center Publ., 2025, pp. 217--229.
\end{thebibliography}

\end{document}